\documentclass{amsart}
\usepackage{amsmath,amsthm,amssymb}

\makeatletter
    
    \@addtoreset{equation}{section}
  \makeatother

\newtheorem{definition}{Definition}[section]

\newtheorem{theorem}[definition]{Theorem}
\newtheorem{lemma}[definition]{Lemma}
\newtheorem{corollary}[definition]{Corollary}
\newtheorem{remark}{Remark}[section]

\newcommand{\re}{\mathbb R} 

\DeclareMathOperator{\diver}{div}
\DeclareMathOperator{\rot}{rot}

\pagebreak

\title[Liouville-type theorems for the Navier-Stokes equations]{
Finite energy of generalized suitable weak solutions to the
Navier-Stokes equations
and Liouville-type theorems in two dimensional domains}
\author{Hideo Kozono, Yutaka Terasawa and Yuta Wakasugi}

\address[H. Kozono]{Department of Mathematics, Faculty of Science and Engineering,
Waseda University, Tokyo 169--8555, Japan}
\email[H. Kozono]{kozono@waseda.jp}
\address[Y. Terasawa]{Graduate School of Mathematics, Nagoya University,
Furocho Chikusaku Nagoya 464-8602, Japan}
\email[Y. Terasawa]{yutaka@math.nagoya-u.ac.jp}
\address[Y. Wakasugi]{Department of Engineering for Production and Environment,
Graduate School of Science and Engineering,
Ehime University,
3 Bunkyo-cho, Matsuyama, Ehime 790-8577, Japan}
\email[Y. Wakasugi]{wakasugi.yuta.vi@ehime-u.ac.jp}
\begin{document}
\begin{abstract}
Introducing a new notion of {\it generalized suitable weak solutions}, 
we first prove validity of the energy inequality for such a class of weak solutions to the 
 Navier-Stokes equations in the whole space $\re^n$.  
Although we need certain growth condition on the pressure, 
we may treat the class even with infinite energy quantity except for the initial velocity.  
We next handle the equation for vorticity in 2D unbounded domains. 
Under a certain condition on the asymptotic behavior at infinity, 
we prove that the vorticity and its gradient of solutions are 
both globally square integrable. 
As their applications, Loiuville-type theorems are obtained.      
\end{abstract}

\maketitle
\section{Introduction}
\footnote[0]{2010 Mathematics Subject Classification. 35Q30; 35B53; 76D05}

We consider the Cauchy problem for the
Navier-Stokes equations
\begin{align}
\label{ns}
	\left\{ \begin{array}{ll}
	v_t - \Delta v + (v\cdot\nabla) v + \nabla p = 0,
		&(x,t)\in \re^n\times (0,T),\\
	\diver v = 0,
		&(x,t)\in \re^n\times (0,T),\\
	v(x,0) = v_0(x),&x\in \re^n.
	\end{array}\right.
\end{align}
Here
$v = v(x,t) = (v^1(x,t), \ldots, v^n(x,t))$
and
$p = p(x,t)$
denote the velocity and the pressure, respectively, 
while 
$v_0 (x) = (v_0^1(x), \ldots, v_0^n(x))$ stands for the given initial velocity. 
Let the initial data
$v_0$
belong to
$L^2_{\sigma}(\re^n)$,
which is the closure of
$C_{0,\sigma}^{\infty}(\re^n)$, 
compactly supported $C^\infty$-solenoidal vector functions,   
with respect to the $L^2$-norm.
We recall that a measurable function
$v$ on $\re^n \times (0,T)$
is a weak solution of the Leray-Hopf class to \eqref{ns} if
$v\in L^{\infty}(0,T;L^2_{\sigma}(\re^n)) \cap L^2_{loc}([0,T);H^1_{\sigma}(\re^n))$ 
and if $v$ satisfies (\ref{ns}) in the sense that   
\[
	\int_0^T \left\{ -(v,\partial_{\tau}\Phi) + (\nabla v, \nabla \Phi)
		+ (v \cdot \nabla v, \Phi) \right\} d\tau
	= (v_0, \Phi(0))
\]
holds for all 
$\Phi \in H^1_0([0,T);H^1_{\sigma}(\re^n)\cap L^n(\re^n))$. 
For every weak solution $v(t)$ of the Leray-Hopf class to \eqref{ns} , 
it is shown by Prodi \cite{Pr} and Serrin \cite{Se63} that, 
after a redefinition of its value of $v(t)$ on a set of measure zero in the time interval 
$[0, T]$, $v(\cdot, t)$ is 
continuous for $t$ in the weak topology of $L^2_{\sigma}(\re^ n)$. 
See also Masuda \cite[Proposition 2]{Ma}.    
\par 
Serrin \cite{Se63} proved that if
$v$
is a weak solution of the Leray-Hopf class to \eqref{ns} and if 
$v \in L^s(0,T;L^q(\re^n))$ for $\frac{3}{q}+ \frac{2}{s} \le 1$ 
with some $q>3, s>2$,
then the energy identity
\begin{align}
\label{ei}
	\| v (t) \|_{L^2}^2 + 2\int_{0}^t \| \nabla v(\tau) \|_{L^2}^2 d\tau
		= \| v_0 \|_{L^2}^2 
		\quad (0\le t < T)
\end{align}
is valid.
Shinbrot \cite{Sh74} also proved that the same conclusion holds  
under another assumption for some   
$s>1, q\ge 4$ such that $\frac{2}{q} + \frac{2}{s} \le 1$.
Taniuchi \cite{Ta97} further extended these results to
\begin{align*}
	&\frac{2}{q}+\frac{2}{s} \le 1,\quad  \frac{3}{q}+\frac{1}{s} \le 1\quad   (n=3),\\
	&\frac{2}{q}+\frac{2}{s} \le 1,\quad  q\ge 4\quad  (n \ge 4).
\end{align*}
Recently, Farwig-Taniuchi \cite{FaTa} obtained a new class by means of domains of fractional 
powers of the Stokes operator in general unbounded domains in  $\re^3$.    
\par
\bigskip
In this paper, we give a new condition which ensures
the energy inequality, and as its application,
several Liouville-type theorems are established.  
Let us first introduce our definition of a generalized suitable weak solution.
\begin{definition}[{\it Generalized suitable weak solution}]\label{def1}
Let $v_0 \in L^2_{\sigma}(\re^n)$.
We say that the pair $(v,p)$ of measurable functions on $\re^n \times (0,T)$ is a 
generalized suitable weak solution 
of the Navier-Stokes equations \eqref{ns} if
\begin{itemize}
\item[(i)]
$v \in L^3_{loc} (\re^n \times [0,T))$,
$\nabla v \in L^2_{loc}(\re^n \times [0,T))$
and
$p \in L^{3/2}_{loc}(\re^n \times [0,T))${\rm ;} 
\item[(ii)]
For every compact subset $K \subset \re^n$,
$v(\cdot, t)$ is continuous for $t \in [0, T)$ in the weak topology of $L^2(K)$ 
and is strongly continuous in $L^2(K)$ at $t = 0$,
that is,
\begin{eqnarray*}
&& 
\int_K v(x, \cdot)\cdot \varphi(x)dx \in C([0, T)) 
\quad 
\mbox{for all $\varphi \in L^2(K)$}, \\
&& 
\lim_{t \to 0+} \int_K |v(x,t)-v_0(x)|^2 dx = 0; 
\end{eqnarray*}
\item[(iii)]
The pair $(v,p)$ satisfies the Navier-Stokes equations \eqref{ns} 
in the sense of distributions in $\re^n\times (0, T)${\rm ;} 
\item[(iv)]
The pair $(v,p)$ fulfills the generalized energy inequality
\begin{align}%
\label{gene_en_ineq}
	2 \int_0^T \int_{\re^n} |\nabla v|^2 \Phi dxdt
	&\le \int_0^T \int_{\re^n}
		\left[ |v|^2 ( \Phi_t + \Delta \Phi) + (|v|^2 + 2p ) v\cdot \nabla \Phi \right]
		dxdt
\end{align}%
for any nonnegative test function
$\Phi \in C_0^{\infty}(\re^n \times (0,T))$.
\end{itemize}
\end{definition}%
\begin{remark} 
{\rm (i)} Caffarelli-Kohn-Nirenberg \cite{CaKoNi} first introduced the notion of 
a suitable weak solution and proved the partial regularity and the Hausdorff dimension of singularities for such weak solutions.  
In comparison with the suitable weak solution given by \cite{CaKoNi}, we assume neither 
finite energy $\sup_{0 < t < T}\|v(t)\|_{L^2}^2 < \infty$ nor finite dissipation 
$\int_{0}^T\|\nabla v(\tau)\|_{L^2}^2d\tau < \infty$.  
Furthermore, we impose on  
the pressure $p$ only local $L^{\frac32}$-bound in $\re^n \times (0, T)$, 
while they \cite{CaKoNi} assume such a global bound as  
$p \in L^{\frac 54}(\re^3 \times (0, T))$ for $n=3$.  
\par
{\rm (ii)} A similar notion to our generalized suitable weak solution was considered by 
Lemari\'e-Rieusset \cite[Chapter 32]{Le} who constructed the local Leray solution 
based on the uniformly local $L^2$-space.              
\end{remark}
Our first main result is the following.

\begin{theorem}\label{thm1}
Let $n\ge 2$,
$v_0\in L^2_{\sigma}(\re^n)$
and let the pair 
$(v,p)$ be
a generalized suitable weak solution of \eqref{ns}.
Suppose that there exist
$q_1, q_2, r_1, r_2$
satisfying
\begin{align}
\label{rq1}
	3\le q_1\le \frac{3n}{n-1},\ 3\le r_1 \le \infty
	\quad \mbox{and}\quad (q_1, r_1) \neq \left( \frac{3n}{n-1}, \infty \right),\\
\label{rq2}
	2 \le q_2 \le \frac{2n}{n-2},\ 2\le r_2 \le \infty
	\quad \mbox{and}\quad
		\begin{cases}
			(q_2, r_2) \neq \left( \frac{2n}{n-2}, \infty \right)
		&(n\ge 3),\\
			q_2 \neq \infty
		&(n=2)
	\end{cases}
\end{align}
such that
$v \in L^3(0,T;L^{q_1, r_1}(\re^n)) \cap L^2(0,T; L^{q_2, r_2}(\re^n))$.
We also assume that the pressure $p$ satisfies
\begin{align}%
\label{ass_p}
	\frac{1}{|B_{|x|/2}(x)|} \int_{B_{|x|/2}(x)} p(y,t) dy = o(|x|)
\quad
\mbox{as $|x| \to \infty$}
\end{align}%
for almost every $t\in (0,T)$. ($B_R(x)$ denotes the ball 
centered at $x \in \re^n$ with radius $R>0$. )
Then, we have that 
\[
	v \in L^{\infty}(0,T;L^2_{\sigma}(\re^n)) \cap L^2 (0,T; \dot{H}^1_{\sigma}(\re^n))
\]
and that 
\[
	\| v(t) \|_{L^2}^2 +2\int_0^t \| \nabla v(\tau ) \|_{L^2}^2 d\tau
	\le \| v_0 \|_{L^2}^2
\]
for all $t\in (0,T)$.
\end{theorem}

\begin{remark} 
{\rm (i)} 
Our proof of Theorem \ref{thm1} enables us to show that 
if the pair $(v,p)$ is a smooth solution with such bounds as 
(\ref{rq1}) and (\ref{rq2}) in Theorem \ref{thm1}
and if $p$ behaves at infinity like (\ref{ass_p}),   
then we have the energy identity (\ref{ei}). 
\par
{\rm (ii)} By (\ref{rq1}) and (\ref{rq2}), it holds that  
$$
\frac23 + \frac{n}{q_1} \ge \frac{n+1}{3} > 1, 
\quad
\frac22 + \frac{n}{q_2} \ge \frac{n}{2} > 1 
\quad
\mbox{for $n \geq 3$}. 
$$
This implies that from a viewpoint of local singularities,  
our class of generalized suitable weak solutions satisfying  (\ref{rq1}) and (\ref{rq2}) 
is larger than Serrin's scaling invariant one $L^{\alpha}(0, T; L^q(\re^n))$ 
for $2/\alpha + n/q =1$ with $n \le q \le \infty$ although we impose 
on those weak solutions a rapid decay property at the spacial infinity.  
See also Remark \ref{rem:1.3}(ii) below.   
Moreover, if $w$ is a Leray-Hopf  weak solution of (\ref{ns}) in Serrin's scaling invariant class, 
then the generalized suitable weak solution $v$ with (\ref{rq1}) and (\ref{rq2}) 
fulfills that $v\equiv w$ on $\re^n\times [0, T)$.  
\par
{\rm (iii)} Besides the energy identity (\ref{ei}), there is another notion of 
the strong energy inequality which means that 
\begin{equation}\label{se}
	\| v(t) \|_{L^2}^2 +2\int_s^t \| \nabla v(\tau ) \|_{L^2}^2 d\tau
	\le \|v(s)\|_{L^2}^2
\end{equation}
for almost all $0 \le s < T$, 
including $s = 0$ and all $t >0$ such that $s\le t \le T$.  
The importance of the strong energy inequality was pointed out by Masuda \cite{Ma}.  
For every $v_0 \in L^2_{\sigma}(\re^n)$, the existence of the weak solution $v$ 
in the Leray-Hopf class satisfying (\ref{se}) 
was proved by Leray \cite{Ler} for $n =3$ and by Kato \cite{Ka} for $n=4$, respectively.  
However, it seems difficult to obtain the corresponding result to the higher dimensional case 
for $n \ge 5$.   
In addition to the condition (ii) of Definition \ref{def1}, if we assume that 
\begin{equation}\label{as}
\lim_{t \to s+0}\int_K |v(x, t) - v(x, s)|^2 dx = 0
\end{equation}
for almost all $0 \le s < T$, including $s =0$, 
then our proof of Theorem \ref{thm1} enables us to see that $v$ satisfies the strong energy inequality (\ref{se}). 
Hence, under such an additional hypothesis as (\ref{as}), 
our generalized suitable weak solution of (\ref{ns}) with (\ref{rq1}) and (\ref{rq2}) 
satisfies Leray's structure theorem(see e.g., Kato \cite[Section 4, 2]{Ka}).        
\par
{\rm (iv)}
The condition \eqref{ass_p} is not restrictive.
Indeed, if $p$ satisfies
$p(x,t) = o(|x|)$ as $|x| \to \infty$ for almost every $t \in (0, T)$, 
then we have \eqref{ass_p}.
Also, if $p$ satisfies
$p \in L^s(0,T ; L^{q,r}(\re^n))$ with some
$s,r \in [1,\infty]$ and $q \in (1,\infty]$, then \eqref{ass_p} holds. 

\end{remark}
\par
An immediate consequence of this theorem is the following Liouville-type theorem. 
\begin{corollary} \label{cor1.2}
Let $n \ge 2$, and let $v_0 \equiv 0$ in $\re^n$.  
Suppose that the pair 
$(v,p)$ is a 
generalized suitable weak solution of \eqref{ns}.
If $p$ satisfies \eqref{ass_p} and if
$v \in L^3(0, T; L^{q_1, r_1}(\re^n))\cap L^2(0, T; L^{q_2, r_2}(\re^n))$ 
for such $(q_1, r_1)$ and $(q_2, r_2)$ as in (\ref{rq1}) and (\ref{rq2}), 
respectively, then it holds that
$v(x, t)\equiv 0$ on $\re^n\times (0, T)$. 
\end{corollary}
%
%
We next deal with the exponents $(q_1, r_1)$ and $(q_2, r_2)$ in the marginal case of 
(\ref{rq1}) and (\ref{rq2}).  
 
\begin{theorem}\label{thm2}
Let
$n \ge 2$,
$v_0\in L^2_{\sigma}(\re^n)$
and let the pair 
$(v,p)$ be a generalized suitable weak solution of \eqref{ns}.
Suppose that
there exist
$q_1, q_2, r_1, r_2$
satisfying
\[
	3 \le q_1 \le \frac{3n}{n-1},\ 2\le q_2 \le \frac{2n}{n-2},\ 
	3\le r_1 \le \infty, \ 2\le r_2 \le \infty
\]
and
\begin{align}
\tag{Case\,1}
	(q_1, r_1) = \left( \frac{3n}{n-1}, \infty \right),\quad
	\begin{cases}
		(q_2, r_2) \neq \left( \frac{2n}{n-2}, \infty \right)
		& (n \ge 3),\\
		q_2 \neq \infty
		& (n =2),
	\end{cases}\\
\tag{Case\,2}
	(q_1, r_1) \neq \left( \frac{3n}{n-1}, \infty \right),\quad
	\begin{cases}
		(q_2, r_2) = \left( \frac{2n}{n-2}, \infty \right)
		& (n \ge 3),\\
		q_2 = \infty
		& (n =2),
	\end{cases}\\
\tag{Case\,3}
	(q_1, r_1) = \left( \frac{3n}{n-1}, \infty \right),\quad
	\begin{cases}
		(q_2, r_2) = \left( \frac{2n}{n-2}, \infty \right)
		& (n \ge 3),\\
		q_2 = \infty
		& (n =2)
	\end{cases}
\end{align}
such that
$v \in L^3(0,T;L^{q_1,r_1}(\re^n))\cap L^2(0,T;L^{q_2,r_2}(\re^n))$.
We also assume that the pressure $p$ satisfies (\ref{ass_p}). 
Then, we have that 
\[
	v \in L^{\infty}(0,T;L^2_{\sigma}(\re^n)) \cap L^2 (0,T;\dot{H}^1_{\sigma}(\re^n))
\]
and that
\begin{align}
\label{eneq}
	\| v(t) \|_{L^2}^2 + 2\int_0^t \| \nabla v(\tau ) \|_{L^2}^2 d\tau
	\le \| v_0 \|_{L^2}^2 + C_0 V_v(t)
\end{align}
holds for all $t\in (0,T)$ with some absolute constant $C_0$, where
\[
	V_v(t) = \begin{cases}
	\| v \|_{L^3(0,t;L^{q_1,r_1})}^3
	& ({\rm Case\,1}),\\
	\| v \|_{L^2(0,t;L^{q_2,r_2})}^2
	& ({\rm Case\,2}),\\
	\| v \|_{L^3(0,t;L^{q_1,r_1})}^3 + \| v \|_{L^2(0,t;L^{q_2,r_2})}^2
	& ({\rm Case\,3}).
	\end{cases}
\]
\end{theorem}
\par
\bigskip
Similarly to Corollary \ref{cor1.2}, we have also the following Liouville-type theorem: 
\begin{corollary}\label{cor1.4} 
Let $n\ge 2$, and let $v_0 \equiv 0$ in $\re^n$.  
Suppose that the pair 
$(v,p)$ is a generalized suitable weak solution of (\ref{ns}).
We assume that
$p$ satisfies \eqref{ass_p} and that 
$v\in L^3(0, T; L^{q_1, r_1}(\re^n)) \cap  L^2(0, T; L^{q_2, r_2}(\re^n))$ 
for such $(q_1, r_1)$ and $(q_2, r_2)$ as in the Cases 1, 2 and 3 in Theorem \ref{thm2}.  
If there exists $\delta \in (0,1/C_0)$
such that
\[
V_v(t_0) 
\le
\delta \left( \| v (t_0) \|_{L^2}^2 + 2\int_0^{t_0} \| \nabla v (\tau ) \|_{L^2}^2 d\tau \right)
\]
for some $t_0 \in (0, T)$, then it holds that
$v(x,t)\equiv0$ on $\re^n \times [0,t_0]$.
\end{corollary}       

\begin{remark}\label{rem:1.3}
{\rm (i)}  
The estimate \eqref{eneq} is invariant under the scaling transformation
$v_{\lambda}(x,t) = \lambda v(\lambda x, \lambda^2 t)$ 
with $\lambda >0$.
Indeed, if $v$ satisfies the estimate (\ref{eneq}) for some $t \in (0, T)$, 
then it holds that 
$$
	\|v_\lambda(t/\lambda^2)\|_{L^2}^2
	+ 2 \int_0^{t/\lambda^2}\|\nabla v_\lambda(\tau)\|_{L^2}^2 d\tau
	\le \|v_{0, \lambda}\|_{L^2}^2 + C_0V_{v_{\lambda}}(t/\lambda^2) 
$$ 
for all $\lambda > 0$.
\par
{\rm (ii)} In comparison with the result of Taniuchi \cite{Ta97},
even for the energy inequality,
Theorem \ref{thm1} requires
stronger integrability of $v$ at the spatial infinity.  
On the other hand, we do not need to impose on $v$ the finite energy and dissipation like  
\begin{equation}\label{energy}
v\in L^{\infty}(0,T;L^2_{\sigma}(\re^n)) \cap L^2_{loc}([0,T);H^1_{\sigma}(\re^n)),
\end{equation}
while \cite{Ta97} requires such a property as (\ref{energy}).   
\end{remark}
%
\par
\bigskip
Next, we consider the initial-boundary value problem for the
Navier-Stokes equations in a two-dimensional domain
$\Omega \subset \re^2$:
\begin{align}
\label{ns2}
	\left\{ \begin{array}{ll}
	v_t - \Delta v + (v\cdot\nabla) v + \nabla p = 0,
		&(x,t)\in \Omega \times (0,T),\\
	\diver v = 0,
		&(x,t)\in \Omega \times (0,T),\\
	v(x,t) = 0,&x\in \partial \Omega,\\
	v(x,0) = v_0(x),&x\in \Omega.
	\end{array}\right.
\end{align}
Here,
$\Omega$
is assumed to be 
$\re^2$,
an exterior domain with smooth boundary
or perturbed half-space with smooth boundary.

For the half-plane $\re^2_+$,
Giga \cite{Gi13} considered an ancient solution
$v$
on
$\re^2_+ \times (-\infty,0)$
and proved
the Liouville type theorem $\omega \equiv 0$
if $v, \nabla v$ are bounded in $\re^2_+ \times (-\infty, 0)$, 
if $\omega(x,t) = \rot v(x,t) \ge 0$ in $\re^2_+ \times (-\infty, t_0)$
with some $t_0 < \infty$,
and if $v^1(t_0) = v^1(x_1, x_2, t_0)$ fulfills the decay condition
\begin{align*}%
	\lim_{R\to \infty} \sup\left\{
		| v^1(x_1, x_2, t_0) | ; x_1 \in \re , x_2 \ge R \right\} = 0.
\end{align*}%
Giga-Hsu-Maekawa \cite{GiHsMa14}
also obtained the same result
under assumptions such that  
$\sup_{-\infty < t <0} (-t)^{1/2} \| v(t) \|_{L^{\infty}} < \infty$
and such that 
$\omega (x,t) \ge 0$,
while there is restriction of spatial decay neither for $\omega$ nor for $v$.
Their proof is based on the Biot-Savart law and hence,
it seems difficult to apply to the case of general domains. 
Another Liouville-type theorem for the ancient solutions in the 2D half plane has been investigated by 
Seregin \cite{Se}. See also Jia-Seregin-Sver\'ak \cite{JiSeSv} and 
Koch-Nadirashvili-Seregin-Sver\'ak \cite{KoNaSeSv}.  
\par
\bigskip
Our result on global integrability of the vorticity now reads as follows.  
\begin{theorem}\label{thm3} 
Let $\omega_0 = \partial_1v_{0}^2-\partial_2 v_{0}^1 \in L^2(\Omega)$, 
and let
$v \in C^{\infty}(\bar{\Omega} \times [0,T))$
be a bounded solution of \eqref{ns2}.
Assume that
$\omega = \partial_1v^2-\partial_2 v^1$
satisfies
\begin{align}
\label{vol}
	\varepsilon_{\omega} (t) := \lim_{R\to \infty} \left(
		\sup_{(x,\tau) \in (\Omega \cap \{ |x| \ge R\} ) \times (0,t)}
		|x|^{1/2} | \omega (x,\tau) | \right) < \infty
\end{align}
for $t \in (0,T)$, and assume also that  
\begin{align}
\label{pa_om}
	\frac{\partial \omega}{\partial \nu} \omega \in L^1(\partial\Omega \times (0,t))
	\quad \mbox{and}\quad
	\int_0^t
		\int_{\partial \Omega}
		\frac{\partial \omega}{\partial \nu} (x,\tau)
		\omega (x, \tau)
		dS d\tau \le 0
\end{align}
for the same $t \in (0,T)$ as in (\ref{vol}),
where
$\displaystyle{\frac{\partial \omega}{\partial \nu}}$
denotes the normal derivative of
$\omega$.
Then, it holds that 
\begin{equation}\label{vortex}
	\omega \in L^\infty(0, t; L^2(\Omega)),
	\quad \nabla \omega \in L^2(0, t; L^2(\Omega))
\end{equation}
with the estimate 
\begin{align}
\label{eneq_om}
	\| \omega(t) \|_{L^2(\Omega)}^2
	+ 2 \int_0^t \| \nabla \omega(\tau) \|_{L^2(\Omega)}^2 d\tau
		\le \| \omega (0) \|_{L^2(\Omega)}^2
	+ C_0 t \| v \|_{L^{\infty}(\Omega \times (0,t))}
		\varepsilon_{\omega}(t) ^2, 
\end{align}
where $C_0>0$ is an absolute constant.  
\end{theorem}
As an application of Theorem \ref{thm3}, 
we have the following Liouville-type theorem in 2D unbounded domains. 
\begin{corollary}\label{cor1.7}  
Let $v_0 =0$, and let $v$ be a smooth bounded solution of (\ref{ns2}).   
If $\varepsilon_{\omega}(t) =0$ for some $t \in (0, T)$ and if the hypothesis (\ref{pa_om}) 
is satisfied for the same $t \in (0, T)$, then it holds that 
$$
v\equiv 0
\quad
\mbox{on $\Omega \times [0, t]$}.  
$$
\end{corollary}

\begin{remark} 
{\rm(i)} 
When $\Omega = \re^2$, then the estimate \eqref{eneq_om} is invariant under the scaling transformation
$v_{\lambda}(x,t) = \lambda v(\lambda x, \lambda^2 t)$,
$\omega_{\lambda}(x,t) = \lambda^2 \omega(\lambda x, \lambda^2 t)$
with $\lambda >0$. Indeed, if $\omega$ satisfies (\ref{eneq_om}) with $\Omega = \re^2$ 
for some $t \in (0, T)$, then we have that 
\begin{align*}
	&\|\omega_\lambda(t/\lambda^2)\|_{L^2(\re^2)}^2
		+ 2\int_0^{t/\lambda^2}\|\nabla \omega_{\lambda}(\tau)\|_{L^2(\re^2)}^2 d\tau\\
	&\le \|\omega_\lambda(0)\|_{L^2(\re^2)}^2
	+ C_0 \frac{t}{\lambda^2} \|v_{\lambda}\|_{L^{\infty}(\re^2\times(0, t/\lambda^2))} 
	\varepsilon_{\omega_{\lambda}}(t/\lambda^2)^2
\end{align*}
for all $\lambda > 0$. 
\par
{\rm(ii)} Corollary \ref{cor1.7} may be regarded as a Liouville-type theorem on bounded 
solutions to (\ref{ns2}) with an additional asymptotic behavior on vorticity
like \eqref{vol}.
A similar result to that on the stationary Navier-Stokes equations in $\re^3$ has been recently 
obtained by the authors \cite{KoTeWa}.  
See also Galdi \cite{Ga} and Chae \cite{Ch1}, \cite{Ch15}.          
\end{remark}


\section{Proof of Theorems \ref{thm1} and \ref{thm2}}
In this section, we give proofs of Theorems \ref{thm1} and \ref{thm2}.

We first introduce notations used throughout this paper.
In what follows, we shall denote by
$C$ various constants which may change from line to line.
In particular, we denote by $C = C(\ast, \ldots, \ast)$
constants depending only on the quantities appearing in parentheses.
Let $L^p(\re^n)$ be the usual Lebesgue space equipped with the norm
\begin{align*}%
	\| f \|_{L^p} := \left( \int_{\re^n} |f(x)|^p dx \right)^{1/p}\ (1\le p < \infty),\quad
	\| f \|_{L^{\infty}} := {\rm ess\,sup\,}_{x\in \re^n} |f(x)|. 
\end{align*}%
Also, $L^p_{\sigma}(\re^n)$ stands for
the space of solenoidal $L^p(\re^n)$-functions.
We denote by $\dot{H}^1(\re^n)$ the inhomogeneous Sobolev space
\begin{align*}%
	\dot{H}^1(\re^n) :=
	\left\{ f \in \mathcal{D}^{\prime}(\re^n) ;
		\nabla f \in L^2 (\re^n) \right\}
\end{align*}%
equipped with the seminorm $\| \nabla f \|_{L^2}$.
$\dot{H}^1_{\sigma}(\re^n)$ is defined in a similar manner as $L^p_{\sigma}(\re^n)$.
We also define $L^p(\Omega) \ (1\le p \le \infty)$
for an exterior domain $\Omega \subset \re^n$ in the same way.

Moreover, we prepare the definition and basic properties of
Lorentz spaces.
\subsection{Basic properties of Lorentz spaces}
Let $1 \le q < \infty$ and $1\le r \le \infty$.
For a measurable function $f$, we define the rearrangement $f^{\ast}$ by  
\[
	f^{\ast}(t) = \sup \left\{ s \in (0,\infty)
		; \mu\left( \{ x \in \re^n; |f(x)| > s \} \right) > t \right\},
\]
where $\mu$ is the Lebesgue measure on $\re^n$.
\begin{definition}
Let $1 \le q < \infty$ and $1\le r \le \infty$.
We define
$L^{q,r}(\re^n)$
by the set of all measurable functions satisfying
$\| f \|_{L^{q,r}} < + \infty$,
where
\[
	\| f \|_{L^{q,r}} = \begin{cases}
	\displaystyle \left( \frac{r}{q}\int_0^{\infty} t^{r/q} f^{\ast}(t)^r \frac{dt}{t} \right)^{1/r},
   & 1\le r <\infty,\\
	\displaystyle \sup_{t>0} t^{1/q}f^{\ast}(t)
	= \sup_{t>0} t \mu \left( \{x \in \re^n; |f(x)|>t\} \right)^{1/q}, &r=\infty.
	\end{cases}
\]
When $q=\infty$ and $1\le r\le \infty$,
we define $L^{q,r} (\mathbb{R}^n) := L^{\infty}(\mathbb{R}^n)$.
\end{definition}
It is well known that
$L^{q,q}(\re^n)$ coincides with the usual Lebesgue space $L^q(\re^n)$ and
the real interpolation yields the equivalence
$(L^{q_0}(\re^n), L^{q_1}(\re^n))_{\theta, r} = L^{q,r}(\re^n)$, 
where $1<q_0 < q < q_1 < \infty$ and $0<\theta<1$ satisfy
$1/q = (1-\theta)/q_0 + \theta/q_1$
and
$1\le r \le \infty$
(see for example, \cite[Theorem 5.3.1, p.113]{BeLobook}).
It is also well known that
$\| \cdot \|_{L^{q,r}}$ is a quasi-norm, namely,
$\| f + g \|_{L^{q,r}} \le C ( \| f\|_{L^{q,r}} + \| g \|_{L^{q,r}})$
holds for some $C>1$, instead of the usual triangle inequality.
We also use the following dilation property of the norm
$\|\cdot\|_{L^{q, r}}$:
\begin{lemma}\label{lem_di}
Let $1\le q<\infty$, $1\le r \le \infty$ and $f\in L^{q,r}(\mathbb{R}^n)$.
For a parameter $R>0$, we put $f_R(x) = f(x/R)$.
Then it holds that 
$\| f_R \|_{L^{q,r}} = R^{n/q}\|f\|_{L^{q,r}}$.
\end{lemma}
For the proof, see for example, \cite[Proposition  2.1]{KoTeWa}.

We will also use the following lemma. 
\begin{lemma}\label{lem_b}
Let $1\le q<\infty$, $1\le r < \infty$ and $f\in L^{q,r}(\mathbb{R}^n)$.
Then, we have
\[
	\lim_{R\to \infty} \| f \|_{L^{q,r}(|x| > R)} =0.
\]
\end{lemma}

We next recall the H\"older inequality of the Lorentz space. 
\begin{lemma}\label{lem_hol}
Let
$1<q<\infty$ and $1\le r \le \infty$ with $1/q^{\prime} + 1/q =1$ and $1/r^{\prime} + 1/r=1$.
Then pointwise multiplication is a bounded bilinear operator in the following cases (i), (ii) and (iii). 
\begin{itemize}
\item[(i)] from $L^{q,r}(\re^n)\times L^{\infty}(\re^n)$ to $L^{q,r}(\re^n)${\rm ;}
\item[(ii)] from $L^{q,r}(\re^n)\times L^{q^{\prime},r^{\prime}}(\re^n)$ to $L^1(\re^n)${\rm ;} 
\item[(iii)] from $L^{q_1,r_1}(\re^n) \times L^{q_2,r_2}(\re^n)$ to $L^{q,r}(\re^n)$ for
$1< q_1, q_2 < \infty$ with $1/q_1+ 1/q_2 = 1/q$ and $1\le r_1, r_2 \le \infty$ with 
$1/r_1 + 1/r_2 = 1/r$.
\end{itemize}
\end{lemma} 
For the proof, see for example \cite[Proposition 2.3]{Le}.

\subsection{Proof of Theorems \ref{thm1} and \ref{thm2}}
One of the crucial points is the estimate of the pressure.
To this end, we employ the expression of the pressure
\[
	p = \sum_{i,j=1}^n R_iR_j(v_iv_j),
\]
where
$R_i = \partial_{x_i}(-\Delta)^{-1/2}$
($i=1,\ldots, n$) denotes the Riesz transform.
When
$p$
satisfies \eqref{ass_p},
we can justify the above formula
in the following way
and obtain the estimate of $p$.
We also refer the reader to the result by Kato \cite{Ka03} in which
the above representation is given for
$p \in L^1_{loc}([0,T) ; {\rm BMO})$ and $v \in L^{\infty}(\re^n \times (0,T))$.

\begin{lemma}\label{lem_p}
Let the pair 
$(v, p)$ be a generalized suitable weak solution of \eqref{ns}
satisfying
$v \in L^s(0,T;L^{q,r}(\re^n))$
with some
$s\in [2,\infty]$, $q\in (2,\infty)$, $r\in [2,\infty]$
and \eqref{ass_p}.
Then, there exists a function
$\bar{p}(t) \in L^1_{loc}([0,T))$
depending only on $t$
such that the function
$p^{\prime} (x,t) := p(x,t)-\bar{p}(t)$
belongs to
$L^{s/2}(0,T; L^{q/2,r/2}(\re^n))$
with the estimate
\begin{align}
\label{es_p}
	\| p^{\prime} \|_{L^{s/2}(0,T;L^{q/2,r/2}(\re^n))} \le C \| v \|_{L^s(0,T;L^{q,r}(\re^n))}^2.
\end{align}
Moreover, when $s=3$, we have
$\bar{p}(t) \in L^{3/2}_{loc}([0,T))$.
\end{lemma}
\begin{proof}
Taking the divergence to both sides of the equation \eqref{ns}, we have
\[
	-\Delta p = \diver( (v\cdot \nabla)v ) = \sum_{i,j=1}^n \partial_{x_i}\partial_{x_j} (v_i v_j)
\]
in the sense of distribution on $\re^n \times (0,T)$.
Let us define the function $p^{\prime}$ by
\[
	p^{\prime} = \sum_{i,j=1}^n R_iR_j(v_iv_j).
\]
By using the boundedness of the Riesz transform on
$L^{\mu}(\re^n)$ for $1<\mu<\infty$
and the general Marcinkiewicz interpolation theorem
(see \cite[Theorem 5.3.2, p.113]{BeLobook}),
we have the boundedness of $R_i$ on $L^{\mu, \rho}(\re^n)$
for $1<\mu<\infty$ and $\rho \in [1,\infty]$, i.e., 
\[
	\| R_i f \|_{L^{\mu,\rho}} \le C \|f \|_{L^{\mu,\rho}}
\quad
\mbox{for all $f \in L^{\mu, \rho}(\re^n)$}.  
\]
By using this estimate twice, we have
\[
	\| p'(t) \|_{L^{q/2,r/2}} \le C \sum_{i,j=1}^n \| R_j (v_i v_j )(t) \|_{L^{q/2,r/2}}
	\le C \sum_{i,j=1}^n \| v_i v_j(t) \|_{L^{q/2,r/2}}
\]
for $q\in (2,\infty)$, $r \in [2,\infty]$
and $t \in (0,T)$.
Applying Proposition \ref{lem_hol} (iii)
and integrating it over $(0,T)$,
we see that $p^{\prime}$ is subject to the estimate \eqref{es_p}. 

Now it remains to show that the function
$\bar{p}(x,t) := p(x,t) - p^{\prime}(x,t)$
is independent of
$x$.
We first note that
$\bar{p} \in L^1_{loc}(\re^n \times [0,T))$.
It is easy to verify that 
$$
	-\Delta p' = \sum_{i,j=1}^n \partial_{x_i}\partial_{x_j} (v_iv_j)
$$
in the sense of distribution in $\re^n \times (0,T)$.
This implies
$-\Delta \bar{p} = 0$
in the sense of distribution in $\re^n \times (0,T)$.
From this, we see that
for almost every $t\in (0,T)$,
the function $\bar{p}(\cdot, t)$ satisfies
$-\Delta \bar{p}(t) = 0$
in the sense of distribution in $\re^n$.
In fact,
by
$\bar{p} \in L^1_{loc} (\re^n \times (0,T))$,
for any $\varphi \in C_0^{\infty}(\re^n)$ and $\eta \in C_0^{\infty}((0,T))$,
we have
\begin{align*}%
	0 &= \iint_{\re^n \times (0,T)} \bar{p}(x,t) (-\Delta \varphi(x) \eta(t) ) dxdt \\
	&= \int_0^T \left( \int_{\re^n} \bar{p}(x,t) (-\Delta \varphi(x)) dx \right)
		\eta (t) dt.
\end{align*}%
Since $\eta$ is arbitrary, we conclude for almost every $t \in (0,T)$ that
\begin{align*}%
	\int_{\re^n} \bar{p}(x,t) (-\Delta \varphi(x)) dx = 0,
\end{align*}%
which implies
$-\Delta \bar{p}(t) = 0$
in the sense of distribution in $\re^n$.
Hence, by Weyl's lemma, for almost every
$t \in (0,T)$,
the function
$\bar{p}(\cdot,t)$
is of class $C^{\infty}(\re^n)$ and harmonic on $\re^n$. 
By the mean value property, we have for almost every $t \in (0,T)$
$$
	\bar{p}(x,t) = \frac{1}{|B_R(x)|}\int_{B_R(x)} \bar{p}(y,t) dy
$$
for all $x \in \re^n$ and all $R>0$,
where $B_R(x)$ denotes the ball in $\re^n$ centered at $x$ with the radius $R$, and 
$|B_R(x)|$ is its volume.
Hence, taking $R = |x|/2$, 
we obtain from \eqref{ass_p} and Lemma \ref{lem_hol} (ii) that 
\begin{align*}
	|\bar{p}(x,t)| &\le
		\frac{1}{|B_{|x|/2}(x)|}
		\int_{B_{|x|/2} (x)}(|p(y,t)| + |p'(y,t)|)dy \\
	&\le
		\frac{1}{|B_{|x|/2}(x)|} \int_{B_{|x|/2} (x)} |p(y,t)| dy \\
	&\quad
	+ \frac{1}{|B_{|x|/2}(x)|}\|p'(t)\|_{L^{q/2, r/2}(B_{|x|/2}(x))}
		\|\chi_{B_{|x|/2}(x)}\|_{L^{q/(q-2),(r/2)'}} \\
	&\le
		\frac{1}{|B_{|x|/2}(x)|} \int_{B_{|x|/2} (x)} |p(y,t)| dy \\
	&\quad + \frac{C}{|B_{|x|/2}(x)|}
		\|p'(t)\|_{L^{q/2, r/2}(\re^n)}|B_{|x|/2}(x)|^{(q-2)/q} \\
	&= o(|x|) + C \|p'(t) \|_{L^{q/2, r/2}(\re^n)}|x|^{-2n/q} \\
	&= o(|x|)
\end{align*}  
for almost every $t \in (0,T)$ as $|x| \to \infty$.
Here $\chi_{B_R(x)}$ denotes the characteristic function of $B_R(x)$
(see e.g, Stein-Weiss \cite[p.192 (3.10)]{StWe}).   
Now, the classical Liouville theorem states that 
$\bar{p}(x,t) = \bar{p}(t)$,
that is,
$\bar{p}$
is independent of
$x$.
Finally, we prove
$\bar{p} \in L^{3/2}_{loc}([0,T))$ when $s=3$.
Similarly to the above, we compute
\begin{align*}%
	|\bar{p}(t)| &=
		\frac{1}{|B_1(0)|} \int_{B_1(0)} \bar{p}(t) dy \\
	&\le \frac{1}{|B_1(0)|} \int_{B_1(0)}
		( |p(y,t)| + |p'(y,t)| )dy\\
	&\le C ( \| p(t) \|_{L^{3/2}(B_1(0))} + \| p'(t) \|_{L^{q/2, r/2}(B_1(0))} ),
\end{align*}%
which implies
\begin{align*}%
	\| \bar{p} \|_{L^{3/2}(J)}
		\le C( \| p \|_{L^{3/2}(B_1(0) \times J)} + \| p' \|_{L^{3/2}(J; L^{q/2, r/2}(B_1(0)))} )
\end{align*}%
for any compact subset $J \subset [0,T)$.
This completes the proof of Lemma \ref{lem_p}.  
\end{proof}

We prepare a localized energy inequality,
which involves the initial data and is derived from
the generalized energy inequality (see Definition \ref{def1} (vi)).
\begin{lemma}\label{lem_en_ineq}
Let the pair $(v, p)$ be a generalized suitable weak solution of \eqref{ns}.
Then, we have
\begin{align*}%
	&\int_{\re^n} |v(x,t)|^2 \psi(x) dx
	+2 \int_0^t \int_{\re^n} |\nabla v(x,\tau)|^2 \psi(x) dxd\tau \\
	&\le
	\int_{\re^n} |v_0(x)|^2 \psi(x) dx \\
	&\quad +\int_0^t \int_{\re^n}
		\left[ |v(x,\tau)|^2 \Delta \psi(x)
			+ (|v(x,\tau)|^2 + 2p(x,\tau) ) v(x,\tau)\cdot \nabla \psi(x) \right]
		dxd\tau 
\end{align*}%
for all $t\in (0,T)$ and for all nonnegative test functions $\psi \in C_0^{\infty}(\re^n)$.
\end{lemma}
\begin{proof} 
We take a time $t \in (0,T)$ and a nonnegative test function
$\psi \in C_0^{\infty}(\re^n)$,
and fix them.
Let $\varepsilon \in (0, t/2)$ be a small parameter and
let $\eta_{\varepsilon} \in C_0^{\infty}((0,t))$ be a nonnegative test function
such that
$\eta_{\varepsilon}(\tau) = 1$
for $\tau\in [\varepsilon, t-\varepsilon]$,
$\eta_{\varepsilon}$ is increasing (resp. decreasing) on
$(0,\varepsilon)$ (resp. $(t-\varepsilon, t)$)
and
$|\eta_{\varepsilon}^{\prime}(\tau)| \le C/\varepsilon$
for
$\tau \in (0,\varepsilon)\cup (t-\varepsilon, t)$. 
Substituting
$\Phi(x, \tau) = \psi (x) \eta_{\varepsilon}(\tau)$
in the generalized energy inequality \eqref{gene_en_ineq}, we see that
\begin{align*}%
	2 \int_0^t \int_{\re^n} |\nabla v|^2 \psi \eta_{\varepsilon} dxd\tau
	&\le \int_0^t \int_{\re^n}
		\left[ |v|^2 ( \psi \eta_{\varepsilon}^{\prime} + (\Delta \psi) \eta_{\varepsilon})
		+ (|v|^2 + 2p ) v\cdot \nabla \psi \eta_{\varepsilon} \right]
		dxd\tau.
\end{align*}%
Since 
$v\in L^3_{loc}(\re^n\times [0,T))$,
$\nabla v \in L^2_{loc}(\re^n \times [0,T))$
and
$p \in L^{3/2}_{loc}(\re^n \times [0,T))$, implied by  Definition \ref{def1}(i), 
we easily notice that
\begin{align*}%
	\lim_{\varepsilon \to 0} \int_0^t \int_{\re^n} |\nabla v|^2 \psi \eta_{\varepsilon} dxd\tau
	&= \int_0^t \int_{\re^n} |\nabla v|^2 \psi dxd\tau,\\
	\lim_{\varepsilon \to 0} \int_0^t \int_{\re^n} |v|^2 (\Delta \psi) \eta_{\varepsilon} dxd\tau
	&= \int_0^t \int_{\re^n} |v|^2 \Delta \psi  dxd\tau,\\
	\lim_{\varepsilon \to 0} \int_0^t \int_{\re^n}
		(|v|^2 + 2p ) v\cdot \nabla \psi \eta_{\varepsilon} dxd\tau
	&= \int_0^t \int_{\re^n}
		(|v|^2 + 2p ) v\cdot \nabla \psi dxd\tau.
\end{align*}%
Therefore, for the proof it suffices to show that
\begin{align}%
\label{eta}
	\liminf_{\varepsilon \to 0}
		\int_0^t \int_{\re^n} |v|^2 \psi \eta_{\varepsilon}^{\prime} dx d\tau
	&\le \int_{\re^n} |v_0(x)|^2 \psi (x) dx - \int_{\re^n} |v(x,t)|^2 \psi (x) dx.
\end{align}%
It follows from $\eta_{\varepsilon}^{\prime} = 0$ on $(\varepsilon, t-\varepsilon)$ that
\begin{align*}%
	\int_0^t \int_{\re^n} |v|^2 \psi \eta_{\varepsilon}^{\prime} dx d\tau
	= \int_0^{\varepsilon} \int_{\re^n} |v|^2 \psi \eta_{\varepsilon}^{\prime} dx d\tau
	+ \int_{t-\varepsilon}^{t} \int_{\re^n} |v|^2 \psi \eta_{\varepsilon}^{\prime} dx d\tau.
\end{align*}%
Recalling the property (ii) of Definition \ref{def1}
and noting
$\int_0^{\varepsilon} \eta_{\varepsilon}^{\prime}(\tau) d\tau = 1$
and
$|\eta_{\varepsilon}^{\prime}(\tau)| \le C/\varepsilon$ for $\tau \in (0,\varepsilon)$,
we compute
\begin{align*}%
	&\int_0^{\varepsilon} \int_{\re^n}
		|v(x,\tau)|^2 \psi(x) \eta_{\varepsilon}^{\prime}(\tau) dx d\tau
		- \int_{\re^n} |v_0(x)|^2 \psi(x) dx \\
	&= \int_0^{\varepsilon} \left( \int_{\re^n}
		\left( |v(x,\tau)|^2 - |v_0(x)|^2 \right) \psi(x)  dx \right)
			\eta_{\varepsilon}^{\prime}(\tau) d\tau \\
	&\to 0
\end{align*}%
as $\varepsilon \to 0$.
Let us estimate the second term
$\int_{t-\varepsilon}^{t} \int_{\re^n} |v|^2 \psi \eta_{\varepsilon}^{\prime} dx d\tau$.
Using the weak continuity of $v$ in $L^2(K)$ on $[0,T)$,
where $K = {\rm supp\,} \psi$,
we see that
$v \in C_w([0,T), L^2(K, \psi(x)dx) )$.
Indeed, for any $\phi \in L^2(K, \psi (x)dx)$, we have
\begin{align*}%
	( v(\tau), \phi )_{L^2(K, \psi dx)}
	&= ( v(\tau), \phi \psi )_{L^2(K)}
\end{align*}%
and the right-hand side is continuous on $[0,T)$ with respect to $\tau$.
Thus, we calculate
\begin{align*}%
	\| v(t) \|_{L^2(K, \psi(x) dx)}^2
	&\le \liminf_{\substack{\tau \to t \\ \tau < t}} \| v(\tau) \|_{L^2(K, \psi(x)dx)}^2 \\
	&= \lim_{\varepsilon \to 0}
		\left( \inf_{t-\varepsilon < \tau < t} \int_{\re^n} |v(x,\tau)|^2 \psi(x) dx \right) \\
	&\le \limsup_{\varepsilon \to 0}
		\int_{t-\varepsilon}^t ( - \eta_{\varepsilon}^{\prime}(\tau) )
			\left( \int_{\re^n} |v(x,\tau)|^2 \psi(x) dx \right)
			d\tau \\
	& = - \liminf_{\varepsilon \to 0}
		\int_{t-\varepsilon}^t 
			\int_{\re^n} |v(x,\tau)|^2 \psi(x) \eta_{\varepsilon}^{\prime}(\tau) dxd\tau.
\end{align*}%
These estimates prove \eqref{eta}.
\end{proof}

Finally, we verify that the pressure $p$ in Lemma \ref{lem_en_ineq}
can be replaced by
$p^{\prime}$ in Lemma \ref{lem_p}.
\begin{lemma}\label{lem_p_prime}
Let the pair 
$(v, p)$ be a generalized suitable weak solution of \eqref{ns}. 
Assume that $v$ satisfies 
$v \in L^3(0,T;L^{q,r}(\re^n))$
with some
$q\in (2,\infty)$, $r\in [2,\infty]$
and that $p$ satisfies \eqref{ass_p}. 
Let $p^{\prime}$ be as in Lemma \ref{lem_p}.
Then, we have
\begin{align*}%
	\int_0^t \int_{\re^n} p(x,\tau) v(x,\tau) \cdot \nabla \psi(x) dxd\tau
	&= \int_0^t \int_{\re^n} p^{\prime}(x,\tau) v(x,\tau) \cdot \nabla \psi(x) dxd\tau
\end{align*}%
for all $t \in (0,T)$ and all nonnegative test functions 
$\psi \in C_0^{\infty}(\re^n)$. 
\end{lemma}
\begin{proof}
We take a time $t \in (0,T)$ and
a nonnegative test function $\psi \in C_0^{\infty}(\re^n)$,
and fix them.
Since $p^{\prime}(x,\tau) = p(x,\tau) - \bar{p}(\tau)$,
it suffices to show that
\begin{align}%
\label{pbar}
	\int_0^t \int_{\re^n} \bar{p}(\tau) v(x,\tau) \cdot \nabla \psi (x) dxd\tau = 0.
\end{align}%
Here we note that
$\bar{p} \in L^{3/2}_{loc}([0,T))$
holds by Lemma \ref{lem_p} with $s=3$
and hence, the integral of the left-had side makes sense.
Let
$\varepsilon \in (0, t/2)$
and we define a test function
$\bar{p}_{\varepsilon}(\tau) \in C_0^{\infty}((0,t))$
approximating $\bar{p}(\tau)$ and defined by
$\bar{p}_{\varepsilon}(\tau)
= \rho_{\varepsilon} \ast ( \chi_{[\varepsilon, t-\varepsilon]} \bar{p} )$
with the characteristic function $\chi_{[\varepsilon, t-\varepsilon]}$
and a mollifier $\rho_{\varepsilon}$.
Then, noting that
$v$ satisfies
$\nabla \cdot v = 0$
in the distribution sense in $\re^n \times (0,T)$
and that
$\bar{p}_{\varepsilon} \psi \in C_0^{\infty}(\re^n \times (0,t))$,
we have
\begin{align}%
\label{eq_bpep}
	\int_0^t \int_{\re^n} \bar{p}_{\varepsilon}(\tau) v(x,\tau) \cdot \nabla \psi (x) dxd\tau = 0.
\end{align}%
Noting
$\lim_{\varepsilon \to 0} \| \bar{p}_{\varepsilon} - \bar{p} \|_{L^{3/2}(0,t)} = 0$
and $v \in L^3_{loc}( \re^n \times [0,T))$,
we let $\varepsilon \to 0$ to obtain
\begin{align*}%
	\lim_{\varepsilon \to 0}
		\int_0^t \int_{\re^n} \bar{p}_{\varepsilon}(\tau) v(x,\tau) \cdot \nabla \psi (x) dxd\tau
	=
	\int_0^t \int_{\re^n} \bar{p}(\tau) v(x,\tau) \cdot \nabla \psi (x) dxd\tau.
\end{align*}%
This and \eqref{eq_bpep} imply \eqref{pbar}.
\end{proof}

Now we are in a position to prove
Theorems \ref{thm1} and \ref{thm2}.
\begin{proof}[Proof of Theorems \ref{thm1} and \ref{thm2}]
Let the pair $(v, p)$ 
be a generalized suitable weak solution of \eqref{ns}. 
Assume that 
$v \in L^3(0,T;L^{q_1, r_1}(\re^n)) \cap L^2(0,T; L^{q_2, r_2}(\re^n))$ 
for $q_1, q_2, r_1, r_2$ with (\ref{rq1}) and (\ref{rq2}),  and assume also that 
$p$ satisfies (\ref{ass_p}). 
Let
$\psi = \psi(x) \in C_0^{\infty}(\re^n)$
be a test function satisfying
\[
	\psi(x) = \begin{cases} 1,&|x|\le 1,\\ 0,&|x|\ge 2,\end{cases} \quad
	0 \le \psi \le 1.
\]
We define a family
$\{ \psi_R \}$
of cut-off functions with large parameter
$R>0$
by
$\psi_R(x) = \psi(x/R)$.
Applying Lemmas \ref{lem_en_ineq} and \ref{lem_p_prime}, we have
\begin{align*}
	&\int_{\re^n} |v(t)|^2 \psi_R dx
		+ 2\int_0^t \int_{\re^n} |\nabla v |^2 \psi_R dx d\tau \\
	&\quad \le
		\int_{\re^n} |v_0|^2 \psi_R dx
		+ \int_0^t \int_{\re^n} |v|^2 \Delta \psi_R dx d\tau \\
	&\qquad
	+ \int_0^t \int_{\re^n} |v|^2 v \cdot \nabla \psi_R dx d\tau
		+ 2\int_0^t \int_{\re^n} p^{\prime} v \cdot \nabla \psi_R dx d\tau \\
	&\quad =: \int_{\re^n} |v_0|^2 \psi_R dx
		+I_R^{(1)} + I_R^{(2)} + I_R^{(3)}.
\end{align*}
We first estimate
$I_R^{(1)}$ and $I_R^{(2)}$.
>From Lemma \ref{lem_hol},
we obtain
\begin{align*}
	I_R^{(1)} & \le C R^{-2}
		\int_0^t \| |v|^2 \|_{L^{\frac{q_2}{2}, \frac{r_2}{2}}(R\le |x|\le 2R)}
			\| \Delta \psi ( \cdot / R ) \|_{%
				L^{\frac{q_2}{q_2-2}, (\frac{r_2}{2})^{\prime}} } d\tau\\
	&\le CR^{-2+\frac{n(q_2-2)}{q_2}}
			\int_0^t \| v \|_{L^{q_2,r_2}(R\le |x| \le 2R)}^2
				\| \Delta \psi \|_{L^{\frac{q_2}{q_2-2},(\frac{r_2}{2})^{\prime}}}
			d\tau \\
	&\le CR^{-2+\frac{n(q_2-2)}{q_2}} \int_0^t \| v \|_{L^{q_2,r_2}(R\le |x|\le 2R)}^2 d\tau
\end{align*}
and
\begin{align*}
	I_R^{(2)} &\le CR^{-1} \int_0^t \| |v|^3 \|_{L^{\frac{q_1}{3}, \frac{r_1}{3} }(R\le |x|\le 2R)}
		\| \nabla \psi (\cdot /R) \|_{%
				L^{\frac{q_1}{q_1-3}, (\frac{r_1}{3})^{\prime}} } d\tau \\
	&\le C R^{-1+\frac{n(q_1-3)}{q_1}} \int_0^t \| v \|_{L^{q_1,r_1}(R\le |x|\le 2R)}^3
		\| \nabla \psi \|_{L^{%
					\frac{q_1}{q_1-3}, (\frac{r_1}{3})^{\prime}} } d\tau \\
	&= C R^{-1+\frac{n(q_1-3)}{q_1}} \int_0^t \|v\|_{L^{q_1,r_1}(R\le |x|\le 2R)}^3 d\tau.
\end{align*}
Moreover, by Lemma \ref{lem_p}, we have
\begin{align*}
	I_R^{(3)} &\le CR^{-1}
		\int_0^t \| p^{\prime}v \|_{L^{\frac{q_1}{3}, \frac{r_1}{3}}(R\le |x|\le 2R)}
		\| \nabla \psi (\cdot /R) \|_{L^{\frac{q_1}{q_1-3}, ( \frac{r_1}{3} )^{\prime}}} d\tau \\
	&\le C R^{-1+\frac{n(q_1-3)}{q_1}}
		\int_0^t \| p^{\prime} \|_{L^{\frac{q_1}{2}, \frac{r_1}{2} }(R\le |x|\le 2R)}
			\| v \|_{L^{q_1, r_1}(R\le |x|\le 2R)}
			\| \nabla \psi \|_{L^{\frac{q_1}{q_1-3}, ( \frac{r_1}{3} )^{\prime}}} d\tau \\
	&\le C R^{-1+\frac{n(q_1-3)}{q_1}}
		\int_0^t \| p^{\prime} \|_{L^{\frac{q_1}{2}, \frac{r_1}{2} }(R\le |x|\le 2R)}
			\| v \|_{L^{q_1, r_1}(R\le |x|\le 2R)} d\tau.
\end{align*}
Let us first prove Theorem \ref{thm1}.
By the assumptions \eqref{rq1} and \eqref{rq2},
we have
\[
	-2+\frac{n (q_2-2)}{q_2} < 0,\ 2 \le r_2 \le \infty\quad
	\mbox{or}\quad 
	-2+\frac{n (q_2-2)}{q_2} \le 0,\ 2 \le r_2 < \infty.
\]
If $-2+ \frac{n (q_2-2)}{q_2} < 0$, then we obtain
\[
	I_R^{(1)} \le C R^{-2+\frac{n (q_2-2)}{q_2} } \| v \|_{L^2(0,T;L^{q_2, r_2})}^2
		\to 0
\]
as $R$ tends to infinity.
On the other hand, if
$2\le r_2 < \infty$,
the fact that
$v(\tau) \in L^{q_2,r_2}(\re^n)$
for almost all
$\tau \in (0,T)$
and Lemma \ref{lem_b} imply
\[
	\lim_{R\to \infty} \| v(\tau) \|_{L^{q_2,r_2}(R\le |x|\le 2R)}^2 = 0
\]
for almost all
$\tau \in (0,T)$.
Thus, by the Lebesgue convergence theorem  we have that 
$I_R^{(1)} \to 0$
as $R$ tends to infinity.
In a similar way, we also obtain
$I_R^{(2)}, I_R^{(3)} \to 0$
as
$R$
tends to infinity.
Consequently, we conclude that
\[
	\int_{\re^n} |v(t)|^2 dx
		+ 2 \int_0^t \int_{\re^n} |\nabla v |^2 dx d\tau
	\le  \int_{\re^n} |v_0|^2 dx.
\]
This completes the proof of Theorem \ref{thm1}. 
\par
\bigskip
Next, we give the proof of Theorem \ref{thm2}.
For simplicity, we treat only Case 1,
because the other cases are quite similar.
In this case, in the same way as before, we first have
$\lim_{R\to \infty} I_R^{(1)} = 0$.
Concerning $I_R^{(2)}$, we only have
$I_R^{(2)} \le C \| v \|_{L^3(0,T;L^{\frac{3n}{n-1}, \infty})}^3$,
since
$-1 + \frac{n (q_1-3)}{q_1} = 0$.
By Lemma \ref{lem_p}, we deduce
$ \| p^{\prime} \|_{L^{q_1/2, \infty}} \le C \| v \|_{L^{q_1, \infty}}^2$
and, hence
$I_R^{(3)} \le C \| v \|_{L^3(0,T;L^{\frac{3n}{n-1}, \infty})}^3$.
Thus, letting $R \to \infty$, we conclude that 
\begin{align*}%
	\int_{\re^n} |v(t)|^2 dx
		+ 2\int_0^t \int_{\re^n} |\nabla v |^2 dx d\tau
	\le \int_{\re^n} |v_0|^2 dx 
		+ C_0 \| v \|_{L^3(0,T;L^{\frac{3n}{n-1}, \infty})}^3
\end{align*}%
with some absolute constant
$C_0>0$.
This completes the proof of Theorems \ref{thm1} and \ref{thm2}.
\end{proof}

We finish this section with the proof of
Corollaries \ref{cor1.2} and \ref{cor1.4}.
\begin{proof}[Proof of Corollaries \ref{cor1.2} and \ref{cor1.4}]
For simplicity, we only give
the proof of Corollary \ref{cor1.4} with Case 1.
If $v_0 \equiv 0$ and
\begin{align*}%
	\| v \|_{L^3(0,t_0;L^{\frac{3n}{n-1}, \infty})}^3
	\le \delta
	\left( \| v(t_0) \|_{L^2}^2 
		+ 2 \int_0^{t_0} \| \nabla v (\tau) \|_{L^2}^2 d\tau \right)
\end{align*}%
with some $t_0 \in (0,T)$ and $\delta \in (0, 1/C_0)$,
then we have $\nabla v \equiv 0$ on $\re^n \times (0,t_0)$,
namely, $v$ is independent of $x$ on $\re^n \times (0,t_0)$.
Further, since $v \in L^3(0,t_0; L^{q_1, r_1}(\re^n))$, we see that
$v \equiv 0$ on $\re^n \times (0,t_0)$.
\end{proof}

\section{Proof of Theorem \ref{thm3}}
We start with the vorticity equation.
Let $v$ be a smooth solution of \eqref{ns2}.
It is well-known that the vorticity $\omega = \rot v$ satisfies
\begin{align}
\label{eq_om}
	\omega_t - \Delta \omega + (v \cdot\nabla)\omega = 0.
\end{align}
Indeed, this can be easily proved by taking the rotation
to the equation \eqref{ns2}.
\begin{proof}[Proof of Theorem \ref{thm3}]
Let
$\psi = \psi (r) \in C_0^{\infty}(\re)$
be a test function satisfying
\[
	\psi(r) = \begin{cases} 1&|r| \le 1,\\ 0&|r| >2 \end{cases}
\]
and define
$\psi_R(x) := \psi ( \frac{x_1}{R} ) \psi (\frac{x_2}{R} )$.
Multiplying \eqref{eq_om} by
$\psi_R \omega$,
we have
\begin{align*}
	&\frac{d}{dt} \frac{1}{2} \int_{\Omega} |\omega|^2 \psi_R dx
	+ \int_{\Omega} |\nabla \omega|^2 \psi_R dx
	- \int_{\partial \Omega} \frac{\partial \omega}{\partial \nu} \omega \psi_R dx\\
	&\quad + \int_{\Omega} ( \nabla \omega \cdot \nabla \psi_R ) \omega dx
	+ \int_{\Omega} (v \cdot \nabla \omega) \omega \psi_R dx = 0.
\end{align*}
When $\Omega$ is an exterior domain, taking $R>0$ sufficiently large, 
we have
$\nabla \psi_R= 0$ on $\partial \Omega$.
Also, when
$\Omega$
is a perturbed half-space, we have
$\nabla \psi_R =0$
on
$\partial \Omega \cap \{ x \in \partial \Omega ; |x| \le R \}$
and
$\frac{\partial \psi_R}{\partial \nu} = 0$
on
$\{ x \in \partial \re^2_+ ; |x| \ge R \}$
for sufficiently large
$R>0$.
Therefore, we obtain
\begin{align*}
	\int_{\Omega} ( \nabla \omega \cdot \nabla \psi_R ) \omega dx
	&= \frac12 \int_{\partial \Omega} \nu \cdot ( \omega^2\nabla \psi_R ) dS
		- \frac12 \int_{\Omega} \omega^2 \Delta \psi_R dx\\
	&= - \frac12 \int_{\Omega} \omega^2 \Delta \psi_R dx.
\end{align*}
Also, the boundary condition
$v = 0$
on
$\partial \Omega$
implies
\begin{align*}
	\int_{\Omega} (v\cdot \nabla \omega) \omega \psi_R dx
	&= \frac{1}{2} \int_{\partial \Omega} v\cdot\nu \ \omega^2 \psi_RdS
		- \frac{1}{2} \int_{\Omega} v \omega^2 \cdot \nabla \psi_R dx\\
	&= -\frac{1}{2} \int_{\Omega} v \omega^2 \cdot \nabla \psi_R dx.
\end{align*}
Consequently, we have
\begin{align*}
	& \frac{d}{dt} \frac{1}{2} \int_{\Omega} |\omega|^2 \psi_R dx
	+ \int_{\Omega} |\nabla \omega|^2 \psi_R dx
	- \int_{\partial \Omega} \frac{\partial \omega}{\partial \nu} \omega \psi_R dx\\
	&\quad = \frac12 \int_{\Omega} \omega^2 \Delta \psi_R dx
		+ \frac{1}{2} \int_{\Omega} v \omega^2 \cdot \nabla \psi_R dx.
\end{align*}
Integrating both sides of the above identity over $[0,t]$, we have
\begin{align}
\label{enineq}
	&\frac{1}{2}\int_{\Omega} |\omega|^2 \psi_R dx
	+ \int_0^t \int_{\Omega} |\nabla \omega|^2 \psi_R dxd\tau
	- \int_0^t \int_{\partial \Omega} \frac{\partial \omega}{\partial \nu} \omega \psi_R dxd\tau \\
\nonumber
	&\quad \le \frac{1}{2} \int_{\Omega} |\omega (x,0)|^2 \psi_R dx
	+ \frac12 \int_0^t \int_{\Omega} \omega^2 \Delta \psi_R dxd\tau 
	+ \frac{1}{2} \int_0^t \int_{\Omega} v \omega^2 \cdot \nabla \psi_R dxd\tau.
\end{align}
To estimate the right-hand side, we use the assumption \eqref{vol}.
Let
\begin{align*}%
	\varepsilon_{\omega}(t;R) :=
	\sup_{(x,\tau) \in (\Omega \cap \{ |x| \ge R\} ) \times (0,t)}
		|x|^{1/2} | \omega (x,\tau) |.
\end{align*}%
Then, we have
$\lim_{R\to \infty} \varepsilon_{\omega}(t;R) = \varepsilon_{\omega}(t)$.
The second term of right-hand side of \eqref{enineq} is estimated as
\begin{align*}
	\left| \frac12 \int_0^t \int_{\Omega} \omega^2 \Delta \psi_R dxd \tau \right|
	&\le CR^{-3} \int_0^t \int_{\Omega \cap \{ R\le |x| \le 2R\}}
		\varepsilon_{\omega}^2(t;R) dxd\tau\\
	&\le CR^{-1} t \varepsilon_{\omega}(t;R)^2.
\end{align*}
On the other hand, the third term is estimated as
\begin{align*}
	&\left| \frac{1}{2} \int_0^t \int_{\Omega} v \omega^2 \cdot \nabla \psi_R dxd\tau \right| \\
	&\quad \le R^{-2} \| \nabla \psi \|_{L^{\infty}(\re^2)} \int_0^t\int_{\Omega\cap  \{ R\le |x| \le 2R\}}
		\| v \|_{L^{\infty}(\Omega\times (0,T))} \varepsilon_{\omega}(t;R)^2 dxd\tau \\
	&\quad \le C_0 \| v \|_{L^{\infty}(\Omega\times (0,T))}
		t \varepsilon_{\omega}(t;R)^2,
\end{align*}
where
$C_0$ is an absolute constant.
Thus, we obtain
\begin{align*}
	&\frac{1}{2}\int_{\Omega} |\omega|^2 \psi_R dx
	+ \int_0^t \int_{\Omega} |\nabla \omega|^2 \psi_R dxd\tau
	- \int_0^t \int_{\partial \Omega} \frac{\partial \omega}{\partial \nu} \omega \psi_R dxd\tau \\
	&\le \frac{1}{2} \int_{\Omega} |\omega(x,0)|^2 \psi_R dx
	+ CR^{-1} t \varepsilon(t;R)^2
	+ C_0 \| v \|_{L^{\infty}(\Omega\times (0,T))} t\varepsilon_{\omega}(t;R)^2.
\end{align*}
>From the assumption \eqref{pa_om}, we obtain
\begin{align*}%
	\lim_{R\to \infty}
	\int_0^t \int_{\partial \Omega} \frac{\partial \omega}{\partial \nu} \omega \psi_R dxd\tau
	= \int_0^t \int_{\partial \Omega} \frac{\partial \omega}{\partial \nu} \omega dxd\tau
	\le 0.
\end{align*}%
Therefore, letting
$R\to \infty$,
we conclude
\begin{align*}
	&\frac{1}{2}\int_{\Omega} |\omega|^2 dx
	+ \int_0^t \int_{\Omega} |\nabla \omega|^2 dxd\tau \\
	&\quad \le \frac{1}{2} \int_{\Omega} |\omega (x,0)|^2 dx
	+ C_0 \| v \|_{L^{\infty}(\Omega\times (0,T))} t \varepsilon_{\omega}(t)^2.
\end{align*} 
This completes the proof of Theorem \ref{thm3}. 
\end{proof}
\par
\bigskip
{\it Proof of Corollary \ref{cor1.7}}.  
Since $v_0 \equiv 0$ and since $\varepsilon_{\omega}(t) = 0$ for $t\in (0,T)$,
then it follows from (\ref{eneq_om}) that 
\begin{align*}%
	\frac{1}{2}\int_{\Omega} |\omega(x, t)|^2 dx
	+ \int_0^t \int_{\Omega} |\nabla \omega(x,\tau) |^2 dxd\tau \le 0,
\end{align*}%
which implies $\omega \equiv 0$ on $\Omega\times [0, t]$. 
Since $\diver v = 0$, we see that $v(\cdot, \tau)$ is harmonic in $\Omega$ for all 
$\tau \in [0, t]$.  Since $v(\cdot, \tau) = 0$ on $\partial \Omega$ and since 
$v$ is bounded in $\Omega\times(0, T)$, it follows from the maximum principle in unbounded 
domains(see e.g., Berestycki-Caffarelli-Nirenberg \cite[Lemma 2.1]{BeCaNi}) that 
$v(\cdot, \tau) \equiv 0$ on $\Omega$ for all $\tau \in [0, t]$.  
This proves Corollary \ref{cor1.7}.

\end{document}